\documentclass[11pt,a4paper]{article}
\pdfoutput=1
\usepackage{amsmath}
\usepackage{amsfonts}
\usepackage{amssymb}
\usepackage{amsthm}
\usepackage{epsfig}
\usepackage{graphicx}
\usepackage[colorlinks]{hyperref}

\newcommand{\Ib}{\ensuremath{\mathbf{I}}}
\newcommand{\Jb}{\ensuremath{\mathbf{J}}}
\newcommand{\Ii}{\ensuremath{\mathbf{i}}}
\newcommand{\Ij}{\ensuremath{\mathbf{j}}}

\newcommand{\IE}{\ensuremath{\mathbb{E}}}
\newcommand{\IL}{\ensuremath{\mathbb{L}}}
\newcommand{\IV}{\ensuremath{\mathbb{V}}}
\newcommand{\IR}{\ensuremath{\mathbb{R}}}
\newcommand{\IN}{\ensuremath{\mathbb{N}}}
\newcommand{\pc}{\ensuremath{p_{c}(d)}}
\newcommand{\IP}{\ensuremath{\mathbb{P}}}

\newcommand{\DD}{\ensuremath{\mathcal{D}}}
\newcommand{\CC}{\ensuremath{\mathcal{C}}}
\newcommand{\JJ}{\ensuremath{\mathcal{J}}}

\newcommand{\FF}{\ensuremath{\mathcal{F}}}
\newcommand{\HH}{\ensuremath{\mathcal{H}}}

\newcommand{\QQ}{\ensuremath{\mathcal{Q}}}
\newcommand{\LL}{\ensuremath{\mathcal{L}}}
\newcommand{\BB}{\ensuremath{\mathcal{B}}}
\newcommand{\UU}{\ensuremath{\mathcal{U}}}

\newcommand{\Var}{\ensuremath{\IV\text{ar}}}
\newcommand{\Dfat}{\ensuremath{D_{\text{fat}}}}
\newcommand{\Dfatc}{\ensuremath{D_{\text{fat}}^{c}}}
\newcommand{\Dfatd}{\ensuremath{D_{\text{fat}}^{d}}}
\newtheorem{theorem}{Theorem}[section]
\newtheorem{lemma}[theorem]{Lemma}

\newtheorem{proposition}[theorem]{Proposition}

\theoremstyle{remark}
\newtheorem{remark}[theorem]{Remark}
\newtheorem{openproblem}[theorem]{Open problem}

\makeatletter
	\let\@fnsymbol\@alph
\makeatother

\begin{document}

\title{Fat fractal percolation and $k$-fractal percolation}
\author{Erik I.\ Broman\footnote{Matematiska Institutionen, Uppsala Universitet, Box 480, 751 06 Uppsala, Sweden; e-mail: \texttt{broman\,(at)\,math.uu.se}. Supported by the G\"oran Gustafsson Foundation for Research in Natural Sciences and Medicine.}
\and Tim van de Brug\footnote{Department of Mathematics, VU University Amsterdam, De Boelelaan 1081a, 1081\,HV Amsterdam, The Netherlands; e-mail: \texttt{\char`{t.vande.brug,f.camia\char`}\,(at)\,vu.nl}. Supported by NWO Vidi grant 639.032.916.}
\and Federico Camia\footnotemark[2]
\and Matthijs Joosten\footnote{Department of Mathematics, VU University Amsterdam, De Boelelaan 1081a, 1081\,HV Amsterdam, The Netherlands; e-mail: \texttt{matthijsjoosten\,(at)\,gmail.com}. Supported by NWO grant 613.000.601.}
\and Ronald Meester\footnote{Department of Mathematics, VU University Amsterdam, De Boelelaan 1081a, 1081\,HV Amsterdam, The Netherlands; e-mail: \texttt{r.w.j.meester\,(at)\,vu.nl}.}}
\date{\today}

\maketitle

\begin{abstract}
We consider two variations on the Mandelbrot fractal percolation model. In the {\em $k$-fractal percolation model}, the $d$-dimensional unit cube is divided in $N^d$ equal subcubes, $k$ of which are retained while the others are discarded. The procedure is then iterated inside the retained cubes at all smaller scales. We show that the (properly rescaled) percolation critical value of this model converges to the critical value of ordinary site percolation on a particular $d$-dimensional lattice as $N \to \infty$. This is analogous to the result of Falconer and Grimmett in \cite{FalGri} that the critical value for Mandelbrot fractal percolation converges to the critical value of site percolation on the same $d$-dimensional lattice.
 
In the {\em fat fractal percolation model}, subcubes are retained with probability $p_n$ at step $n$ of the construction, where $(p_n)_{n\geq 1}$ is a non-decreasing sequence with $\prod_{n=1}^\infty p_n > 0$. The Lebesgue measure of the limit set is positive a.s.\ given non-extinction. We prove that either the set of connected components larger than one point has Lebesgue measure zero a.s.\ or its complement in the limit set has Lebesgue measure zero a.s.


\end{abstract}

\section{Introduction}\label{sec introduction}

In \cite{Man} Mandelbrot introduced the following fractal percolation model. Let $N \geq 2, d\geq 2$ be integers and consider the unit cube $[0,1]^{d}$. Divide the unit cube into $N^{d}$ subcubes of side length $1/N$. Each subcube is retained with probability $p$ and discarded with probability $1-p$, independently of other subcubes. The closure of the union of the retained subcubes forms a random subset $D^{1}_{p}$ of $[0,1]^d$. Next, each retained subcube in $D^{1}_{p}$ is divided into $N^d$ cubes of side length $1/N^{2}$. Again, each smaller subcube is retained with probability $p$ and discarded with probability $1-p$, independently of other cubes. We obtain a new random set $D^{2}_{p} \subset D^1_{p}$. Iterating this procedure in every retained cube at every smaller scale yields an infinite decreasing sequence of random subsets $D^1_{p} \supset D_p^2 \supset D_p^3 \supset \cdots$ of $[0,1]^d$. We define the limit set $D_{p} := \bigcap_{n=1}^{\infty} D^{n}_{p}$. We will refer to this model as the Mandelbrot fractal percolation (MFP) model with parameter $p$.

It is easy to extend and generalize the classical Mandelbrot model in ways that preserve at least a certain amount of statistical self-similarity and generate random fractal sets. It is interesting to study such models to obtain a better understanding of general fractal percolation processes and explore possible new features that are not present in the MFP model. In this paper we are concerned with two natural extensions which have previously appeared in the literature, as we mention below. We will next introduce the models and state our main results.

\subsection[k-fractal percolation]{$k$-fractal percolation}
Let $N \geq 2$ be an integer and divide the unit cube $[0,1]^d$, $d \geq 2$, into $N^d$ subcubes of side length $1/N$. Fix an integer $0 < k  \leq N^d$ and retain $k$ subcubes in a uniform way, that is, all configurations where $k$ cubes are retained have equal probability, other configurations have probability 0. Let $D_k^{1}$ denote the random set which is obtained by taking the closure of the union of all retained cubes.
Iterating the described procedure in retained cubes and on all smaller scales yields a decreasing sequence of random sets $D_k^1 \supset D_k^2 \supset D_k^3 \supset \cdots$. We are mainly interested in the connectivity properties of the limiting set $D_{k} := \bigcap_{n=1}^\infty D^{n}_{k}$.
This model was called the \emph{micro-canonical fractal percolation process} by Lincoln Chayes in \cite{Cha} and both \emph{correlated fractal percolation} and \emph{$k$ out of $N^d$ fractal percolation} by Dekking and Don \cite{DekDon}. We will adopt the terms \emph{$k$-fractal percolation} and \emph{$k$-model}.

For $F\subset [0,1]^d$, we say that the unit cube is \emph{crossed by $F$} if there exists a connected component of $F$ which intersects both $\{0\} \times [0,1]^{d-1}$ and $\{1\} \times [0,1]^{d-1}$. Define $\theta(k,N,d)$ as the probability that $[0,1]^d$ is crossed by $D_k$. Similarly, $\sigma(p,N,d)$ denotes the probability that $[0,1]^d$ is crossed by $D_{p}$. Let us define the critical probability $p_c(N,d)$ for the MFP model and the critical threshold value $k_c(N,d)$ for the $k$-model by
\[
 p_c(N,d) := \inf\{ p: \sigma(p,N,d) >0\}, \quad k_{c}(N,d) := \min \{k: \theta(k,N,d)>0\}.
\]
Let ${\mathbb L}^d$ be the $d$-dimensional lattice with
vertex set ${\mathbb Z}^d$ and with edge set given by the adjacency
relation: $(x_1, \ldots, x_d)=x \sim y=(y_1, \ldots, y_d)$
if and only if $x \neq y$, $|x_i-y_i| \leq 1$ for all $i$
and $x_i=y_i$ for at least one value of $i$. Let $p_c(d)$ denote the critical probability for site percolation on 
$\IL^d$. It is known (see \cite{FalGri}) that $p_{c}(N,d) \to p_c(d)$ as $N \to \infty$. We have the following analogous result for the $k$-model.

\begin{theorem}\label{thm critical k}
For all $d \geq 2$, we have that
 \[
 \lim_{N \to \infty} \frac{k_{c}(N,d)}{N^{d}} = p_{c}(d). 
 \]
\end{theorem}

\begin{remark} Note that the choice for the unit cube in the definitions of $\theta(k,N,d)$ and $\sigma(p,N,d)$ (and thus implicitly also in the definitions of $k_c(N,d)$ and $p_c(N,d)$) is rather arbitrary: We could define them in terms of crossings of other shapes such as annuli, for example, and obtain the same conclusion, i.e.\ $k_c(N,d)/N^d \to \pc$ as $N \to \infty$, where $\theta(k,N,d)$ and $k_c(N,d)$ are defined using the probability that $D_k$ crosses an annulus. One advantage of using annuli is that the percolation function $\sigma(p,N,d)$ is known to have a discontinuity at $p_c(N,d)$ for all $N,d$ and any choice of annulus \cite[Corollary 2.6]{BroCam}. (This is known to be the case also when $p_c(N,d)$ is defined using the unit cube if $d=2$ \cite{ChaChaDur, DekMee}, but for $d\geq 3$ it is proven only for $N$ sufficiently large \cite{BroCam08}.) In the present paper we stick to the ``traditional" choice of the unit cube.
\end{remark}

\begin{remark}\label{remark supercritical MFP}
For the MFP model it is the case  that, for $p>\pc$,
\begin{equation}\label{prob to 1}
\sigma(p,N,d) \to 1,  
\end{equation}
as $N \to \infty$. This is part (b) of Theorem 2 in \cite{FalGri}. During the course of the proof of Theorem \ref{thm critical k} we will prove a similar result for the $k$-model, see Theorem \ref{thm supercritical}.
\end{remark}

Next, consider the following generalization of both the $k$-model and the MFP model.
Let $d\geq2, N \geq 2$ be integers and let $Y=Y(N,d)$ be a random variable taking values in $\{0,\ldots,N^d\}$. Divide the unit cube into $N^d$ smaller cubes of side length $1/N$. Draw a realization $y$ according to $Y$ and retain $y$ cubes uniformly. Let $D^1_{Y}$ denote the closure of the union of the retained cubes. Next, every retained cube is divided into $N^d$ smaller subcubes of side length $1/N^2$. Then, for every subcube $C$ in $D^{1}_{Y}$ (where we slightly abuse notation by viewing $D^1_{Y}$ as the set of retained cubes in the first iteration step) draw a new (independent) realization $y(C)$ of $Y$ and retain $y(C)$ subcubes in $C$ uniformly, independently of all other subcubes. Denote the closure of the union of retained subcubes by $D^2_{Y}$. Repeat this procedure in every retained subcube at every smaller scale and define the limit set $D_{Y}:= \bigcap_{n=1}^{\infty}D^n_{Y}$. We will call this model the \emph{generalized fractal percolation model} (GFP model) \emph{with generator $Y$}. Define $\phi(Y,N,d)$ as the probability of the event that $[0,1]^d$ is crossed by $D_{Y}$.

By taking $Y$ equal to an integer $k$, resp.\ to a binomially distributed random variable with parameters $N^d$ and $p$, we obtain the $k$-model, resp.\ the MFP model with parameter $p$. If $Y$ is stochastically dominated by a binomial random variable with parameters $N^d$ and $p$, where $p<p_c(N,d)$, then by standard coupling techniques it follows that $\phi(Y,N,d)=0$. Likewise, if $Y(N,d)$ dominates a binomial random variable with parameters $N^d$ and $p$, where $p>\pc$, then $\phi(Y(N,d),N,d) \geq \sigma(p,N,d) \to 1$ as $N \to \infty$, as mentioned in Remark \ref{remark supercritical MFP}. The following theorem, which generalizes \eqref{prob to 1}, shows that the latter conclusion still holds if for some $p>\pc$, $\IP(Y(N,d) \geq p N^d)\to 1$ as $N \to \infty$.

\begin{theorem}\label{thm generalization}
 Consider the GFP model with generator $Y(N,d)$. Let $p>\pc$. Suppose that $\IP(Y(N,d) \geq pN^{d}) \to 1$ as $N \to \infty$. Then 
\[
\lim_{N \to \infty} \phi(Y(N,d),N,d) = 1.
\]
\end{theorem}

\begin{remark} Observe that by Chebyshev's inequality the condition of Theorem \ref{thm generalization} is satisfied if, for some $p>\pc$, $\IE Y(N,d) \geq p N^d$ for all $N \geq 2$ and $\Var(Y(N,d))/N^{2d} \to 0$ as $N \to \infty$.
\end{remark}

\begin{openproblem} It is a natural question to ask whether a ``symmetric version'' of Theorem \ref{thm generalization} is true. That is, if e.g.\ $\IP(Y(N,d) \leq p N^d) \to 1$ as $N \to \infty$, for some $p<\pc$, implies $\phi(Y(N,d),N,d) \to 0$ as $N \to \infty$. The proof of Theorem \ref{thm generalization} can not be adapted to this situation.
\end{openproblem}

\subsection{Fat fractal percolation}

Let $(p_n)_{n \geq 1}$ be a non-decreasing sequence in $(0,1]$ such that $\prod_{n=1}^{\infty} p_n >0$. We call \emph{fat fractal percolation} a model analogous to the MFP model, but where at every iteration step $n$ a subcube is retained with probability $p_n$ and discarded with probability $1-p_n$, independently of other subcubes. Iterating this procedure yields a decreasing sequence of random subsets 
$D^1_{\text{fat}} \supset D^2_{\text{fat}} \supset D^3_{\text{fat}}\supset \cdots$ and we will mainly study connectivity properties of the limit set $D_{\text{fat}} := \bigcap_{n=1}^\infty D_{\text{fat}}^n$. In \cite{ChaPemPer} it is shown that if $p_n\to 1$ and $\prod_{n=1}^{\infty} p_n=0$, then the limit set does not contain a directed crossing from left to right.

For a point $x \in \Dfat$, let $C_{\text{fat}}^x$ denote its \emph{connected component}:
\[
 C_{\text{fat}}^x := \{y \in \Dfat: y \text{ connected to } x \text{ in } \Dfat\}.
\]
We define the set of ``dust'' points by $D^{d}_{\text{fat}}:= \{ x \in \Dfat: C_{\text{fat}}^x =\{x\} \}$. Define $D^{c}_{\text{fat}}:=\Dfat \setminus \Dfat^d$, which is the union of connected components larger than one point. Let $\lambda$ denote the $d$-dimensional Lebesgue measure. It is easy to prove that $\lambda(\Dfat)>0$ with positive probability, see Proposition \ref{prop expectation}. Moreover, we can show that the Lebesgue measure of the limit set is positive a.s.\ given non-extinction, i.e.\ $\Dfat \neq \emptyset$.

\begin{theorem}\label{thm12}
We have that $\lambda(\Dfat)>0$ a.s.\ given non-extinction.
\end{theorem}

It is a natural question to ask whether both $\Dfatc$ and $\Dfatd$ have positive Lebesgue measure. The following theorem shows that they cannot simultaneously have positive Lebesgue measure.

\begin{theorem}\label{thm lebesgue}
Given non-extinction of the fat fractal process, it is the case that either 
\begin{equation}\label{fat con}
\lambda(\Dfatd) =0 \text{ and } \lambda(\Dfatc)>0 \text{ a.s.}                                                                              
\end{equation}
 or
\begin{equation}\label{fat dust}
\lambda(\Dfatd) >0 \text{ and } \lambda(\Dfatc)=0 \text{ a.s.}
\end{equation}
\end{theorem}

Part (ii) of the following theorem gives a sufficient condition under which \eqref{fat con} holds. Furthermore, the theorem shows that the limit set either has an empty interior a.s.\ or can be written as the union of finitely many cubes a.s.

\begin{theorem}\label{thm charac} We have that
\begin{itemize}
 \item[$(i)$] If $\prod_{n=1}^{\infty} p_n^{N^{dn}} = 0$, then $\Dfat$ has an empty interior a.s.;
\item[$(ii)$] If $\prod_{n=1}^{\infty} p_n^{N^n} > 0$, then $\lambda(\Dfatd) = 0$ a.s.;
\item[$(iii)$] If $\prod_{n=1}^{\infty} p_n^{N^{dn}} > 0$, then $\Dfat$ can be written as the union of finitely many cubes a.s.
\end{itemize}
\end{theorem}

\begin{openproblem} 
Part (ii) of Theorem \ref{thm charac} shows that if $\prod_{n=1}^{\infty} p_{n}^{N^n} > 0$, then \eqref{fat con} holds. However, we do not have an example for which \eqref{fat dust} holds, and we do not know whether \eqref{fat dust} is possible at all.
\end{openproblem}

In two dimensions, we have the following characterizations of $\lambda(\Dfat^c)$ being positive a.s.\ given non-extinction of the fat fractal process.

\begin{theorem} \label{thm pos prob} Let $d=2$. The following statements are equivalent.
\begin{itemize}
 \item[$(i)$] $\lambda(D^c_{\text{fat}}) > 0$ a.s., given non-extinction of the fat fractal process;
\item[$(ii)$] There exists a set $U\subset [0,1]^2$ with $\lambda(U)>0$ such that for all $x,y \in U$ it is the case that $\IP( x \text{ is in the same connected component as } y) >0$;
\item[$(iii)$] There exists a set $U\subset [0,1]^2$ with $\lambda(U)=1$ such that for all $x,y \in U$ it is the case that $\IP( x \text{ is in the same connected component as } y) >0$.
\end{itemize}
\end{theorem}

Let us now outline the rest of the paper. The next section will be devoted to a formal introduction of the fractal percolation processes in the unit cube. We also define an ordering on the subcubes which will facilitate the proofs of Theorems \ref{thm critical k} and \ref{thm generalization} in Section \ref{sec proof k fractal}. In Section \ref{sec main fat} we prove our results concerning fat fractal percolation.

\section{Preliminaries}\label{sec pre}

In this section we set up an ordering for the subcubes of the fractal processes in the unit cube which will turn out to be very useful during the course of the proofs. We also give a formal probabilistic definition of the different fractal percolation models.
We follow \cite{FalGri} almost verbatim in this section; a simple reference to \cite{FalGri} would however not be very useful for the reader, so we repeat some definitions here.

Order $J^d:=\{0,1,\ldots,N-1\}^d$ in some way, say lexicographically by coordinates. For a positive integer $n$, write $J^{d,n}:=\{(\Ii_1,\ldots,\Ii_n): \Ii_j \in J^d, 1 \leq j \leq n \}$ for the set of $n$-vectors with entries in $J^d$. Set $J^{d,0} := \{\emptyset\}$. With $\Ib = (\Ii_1,\ldots,\Ii_n) =((i_{1,1},\ldots,i_{1,d}),\ldots,(i_{n,1},\ldots,i_{n,d}))$ we associate the subcube of $[0,1]^d$ given by
\[
 C(\Ib) = c(\Ib) +[0,N^{-n}]^d,
\]
where 
\[
 c(\Ib) = \left(\sum_{j=1}^{n}N^{-j}i_{j,1},\ldots,\sum_{j=1}^{n}N^{-j}i_{j,d}\right)
\]
and $c(\emptyset)$ is defined to be the origin.
Such a cube $C(\Ib)$ is called a \emph{level-$n$ cube} and we write $|\Ib| = n$. A concatenation of $\Ib \in J^{d,n}$ and $\Ij \in J^d$ is denoted by $(\Ib,\Ij)$, which is in $J^{d,n+1}$. We define the set of indices for all cubes until (inclusive) level-$n$ as $\JJ^{(n)}:=J^{d,0} \cup J^{d,1} \cup \cdots \cup J^{d,n}$ and we order them in the following way. We declare $\Ib=(\Ii_1,\ldots,\Ii_a)<\Ib'=(\Ii_1',\ldots,\Ii_b')$ if and only if
\begin{itemize}
 \item \emph{either} $\Ii_r < \Ii_r'$ (according to the order on $J^d$) where $r \leq \min\{a,b\}$ is the smallest index so that $\Ii_r \neq \Ii_r'$ holds;
\item \emph{or} $a>b$ and $\Ii_r = \Ii_r'$ for $r=1,\ldots,b$.
\end{itemize}
To clarify this ordering we give a short example, see Figure \ref{fig ordering}. Suppose $N=2$, $d=2$ and $J^2$ is ordered by $(1,1)>(1,0)>(0,1)>(0,0)$, then the ordering of $\JJ^{(2)}$ starts with
\begin{eqnarray*}
 \emptyset &>& ((1,1))> ((1,1),(1,1))> ((1,1),(1,0)) \\
&>& ((1,1),(0,1)) >((1,1),(0,0)) > ((1,0)) > \ldots
\end{eqnarray*}

\begin{figure}[!ht]
\begin{center}
\includegraphics[width=7cm]{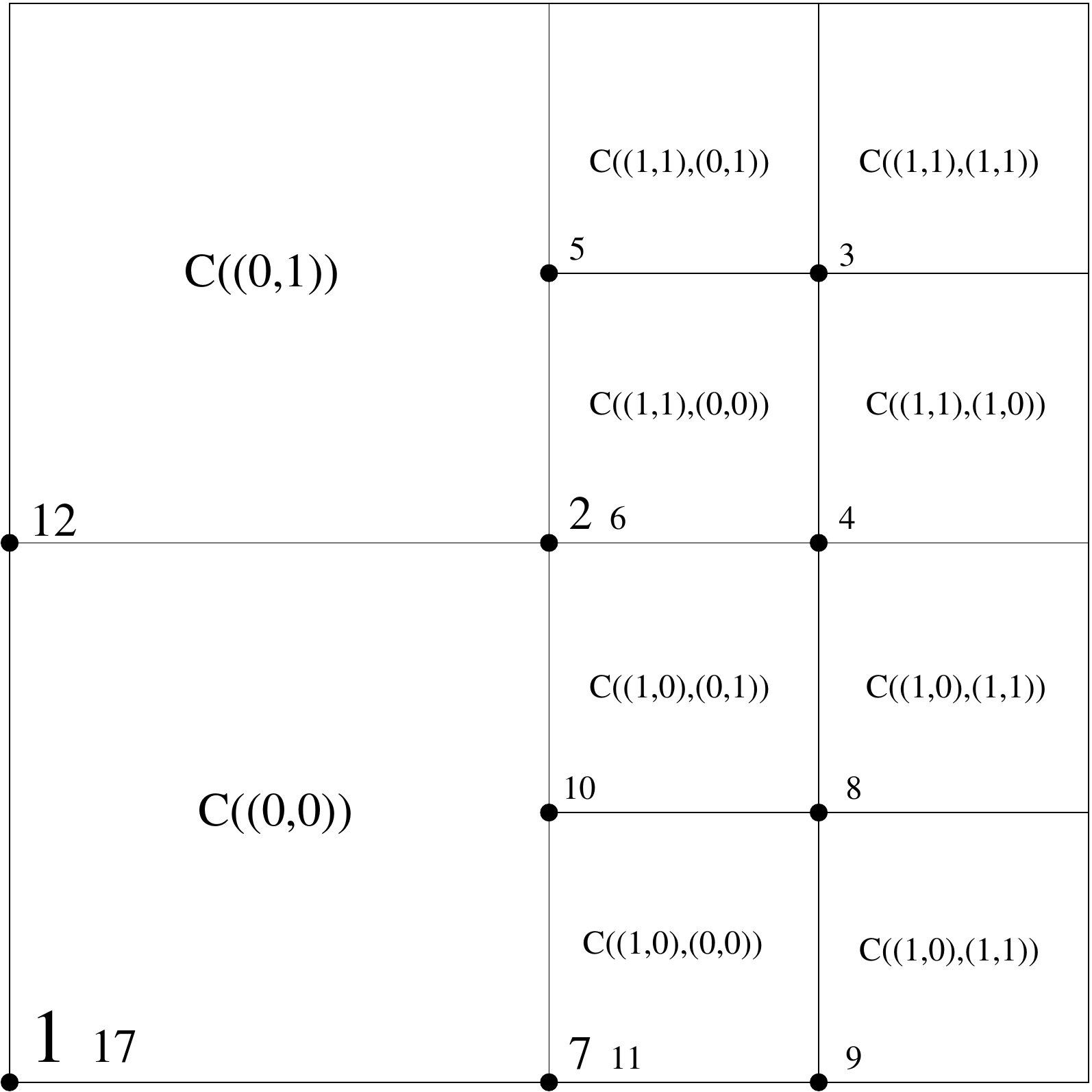}
\caption{Illustration of the ordering of subcubes in $\JJ^{(2)}$, for $N=2$ and $d=2$. A black dot denotes the corner point $c(\Ib)$ of a subcube $C(\Ib)$. The number in the lower left corner of a subcube indicates the rank of the subcube in the ordering: e.g.\ the unit cube, i.e.\ $C(\emptyset)$, has rank 1 and $C((0,0))$ has rank 17.  }
\label{fig ordering}
\end{center}
\end{figure}

We introduce the following formal probabilistic definition of the fractal percolation models. As noted before, the $k$-model and MFP model can be obtained from the GFP model with generator $Y$ by setting $Y \equiv k$, resp.\ $Y$ binomially distributed with parameters $N^d$ and $p \in [0,1]$. Therefore, we only provide a formal probabilistic definition of the GFP model and the fat fractal percolation model. Define the index set $\JJ := \bigcup_{n=0}^{\infty} J^{d,n}$. We define a family of random variables $\{Z_{\text{model}}(\Ib)\}$, where $\Ib \in \JJ$ and -- here as well as in the rest of the section -- ``$\text{model}$'' stands for either $p,\text{fat}$, $k$ or $Y$.
\begin{itemize}
 \item[1.] GFP model with generator $Y$: For every $\Ib \in \JJ$, let $y(\Ib)$ denote a realization of $Y$, independently of other $\Ib'$. We define $J(\Ib)$ as a uniform choice of $y(\Ib)$ different indices of $J^d$, independently of other $J(\Ib')$. For $\mathbf{j} \in J^d$ define
\[
 Z_Y(\Ib,\mathbf{j}) = \left\{
\begin{array}{ll}
1, &  \mathbf{j} \in J(\Ib), \\
0, & \text{otherwise.}
\end{array}
\right.
\]
\item[2.] Fat fractal percolation with parameters $(p_{n})_{n\geq 1}$: For every $\Ib \in \JJ$ and $\mathbf{j} \in J^d$, let $n = |\Ib|$ and define
\[
 Z_{\text{fat}}(\Ib,\mathbf{j}) = \left\{
\begin{array}{ll}
1, &  \text{with probability } p_{n+1}, \\
0, & \text{with probability } 1-p_{n+1},
\end{array}
\right.
\]
independently of all other $Z_{\text{fat}}(\Ib')$.

\end{itemize}

For each $\Ib \in \JJ$ we define the indicator function $1_{\text{model}}(\Ib)$ by
\[
1_{\text{model}}(\emptyset)=1,\quad 1_{\text{model}}(\Ib)=Z_{\text{model}}(\Ii_1) Z_{\text{model}}(\Ii_1,\Ii_2) \cdots Z_{\text{model}}(\Ib),
\]
where $\Ib=(\Ii_1,\Ii_2,\ldots,\Ii_n) \in J^{d,n}$. We retain the subcube $C(\Ib)$ if $1_{\text{model}}(\Ib)=1$ and we write $D_{\text{model}}^n$ for the set of retained level-$n$ cubes. Note that $D_{\text{model}}^1, D_{\text{model}}^2, \ldots$ correspond to the sets informally constructed in the introduction. We denote by $\IP_{\text{model}}$ the distribution of the corresponding model on $\Omega = \{0,1\}^{\CC}$, where $\CC:=\{C(\Ib): \Ib \in \JJ\}$ denotes the collection of all subcubes, endowed with the usual sigma algebra generated by the cylinder events. To simplify the notation, we will drop the subscripts $\text{fat},k,p,Y$ when there is no danger of confusion.

\section[Proofs of the k-fractal results]{Proofs of the $k$-fractal results}\label{sec proof k fractal}

In this section we prove Theorem \ref{thm critical k} and Theorem \ref{thm generalization}. The proof of Theorem \ref{thm critical k} is divided in two parts. First we treat the subcritical case and show that $\liminf_{N \to \infty} k_c(N,d)/N^d \geq \pc$.

\begin{theorem}\label{thm subcritical}
Consider the $k$-model. We have
\[
\liminf_{N \to \infty} k_c(N,d)/N^d \geq \pc.
\]                                           
\end{theorem}
\noindent
In the supercritical case, we prove that the crossing probability converges to 1 as $N \to \infty$. Again, for future reference we state this as a theorem.

\begin{theorem}\label{thm supercritical}
Let $p>\pc$ and let $(k(N))_{N\geq 2}$ be a sequence of integers such that $k(N) / N^d \geq p$, for all $N \geq 2$. We have
\[
 \lim_{N \to \infty} \theta(k(N),N,d) = 1.
\]
\end{theorem}
\noindent
Theorem \ref{thm critical k} follows immediately from these two theorems.

We prove Theorems \ref{thm subcritical} and \ref{thm supercritical} in Sections \ref{sec subcritical} and \ref{sec supercritical}, respectively. In Section \ref{sec generalization} we prove Theorem \ref{thm generalization}, using the idea of the proof of Theorem \ref{thm subcritical} and the result of Theorem \ref{thm supercritical}.

\subsection{Proof of Theorem \ref{thm subcritical}}\label{sec subcritical}

Let $p<\pc$ and consider a sequence $(k(N))_{N \geq 2}$ such that $k(N)/N^{d} \leq p$, for all $N \geq 2$, and $k(N)/N^d \to p$ as $N \to \infty$. Our goal is to show that the probability that the unit cube is crossed by $D_{k(N)}$, is equal to zero for all $N$ large enough. Let $N \geq 2$ and let $D_{p_{0}}$ be the limit set of an MFP process with parameters $p_0$ and $N$, where $p<p_0<\pc$. First, part (a) of Theorem 2 in \cite{FalGri} states that 
\begin{equation}\label{pc subcritical}
\pc \leq p_c(N,d),
\end{equation}
 for all $N$. Hence, the MFP process with parameter $p_0<\pc$ is subcritical. Therefore, a natural approach to prove that the probability that $D_{k(N)}$ crosses the unit cube equals zero for $N$ large enough would be to couple the limit set $D_{k(N)}$ to the limit set $D_{p_{0}}$ in such a way that $D_{k(N)} \subset D_{p_{0}}$. However, a ``direct" coupling between the limit sets $D_{k(N)}$ and $D_{p_{0}}$ is not possible, since with fixed positive probability at each iteration of the MFP process the number of retained subcubes is less than $k(N)$. We therefore need to find a more refined coupling.

The following is an informal strategy of the proof. We will define an event $E$ on which the MFP process contains an infinite tree of retained subcubes, such that each subcube in this tree contains at least $k(N)$ retained subcubes in the tree. Next, we perform a construction of two auxiliary random subsets of the unit cube, from which it will follow that the law of $D_{k(N)}$ is stochastically dominated by the conditional law of $D_{p_0}$, conditioned on the event $E$. In particular, the probability that $D_{k(N)}$ crosses $[0,1]^d$ is less than or equal to the conditional probability that $D_{p_{0}}$ crosses the unit cube, given $E$. The latter probability is zero for $N$ large enough, since the event $E$ has positive probability for $N$ large enough and the MFP process is subcritical.

Let us start by defining the event $E$. Consider an MFP process with parameters $p_0$ and $N$. For notational convenience we call the unit cube the level-0 cube. A level-$n$ cube, $n \geq 0$, is declared \emph{0-good} if it is retained and contains at least $k(N)$ retained level-$(n+1)$ subcubes. (We adopt the convention that $[0,1]^d$ is automatically retained.)
Recursively, we define the notion \emph{$m$-good}, for $m\geq 0$. A level-$n$ cube, for $n\geq 0$, is $(m+1)$-good if it is retained and contains at least $k(N)$ $m$-good subcubes. We say that the unit cube is \emph{$\infty$-good} if it is $m$-good for every $m\geq0$. Define the following events
\begin{eqnarray}
 E_m := \{ [0,1]^d  \text{ is } m\text{-good}\}, \label{E_m} \nonumber \\
 E:= \{[0,1]^d \text{ is } \infty\text{-good}\}. \label{E}
\end{eqnarray}

The following lemma states that we can make the probability of $E$ arbitrary close to 1, for $N$ large enough. In particular, $E$ has positive probability for large enough $N$, which will be sufficient for the proof of Theorem \ref{thm subcritical}.

\begin{lemma}\label{lem high prob}
 Let $p_0 < \pc$. Let $(k(N))_{N\geq 2}$ be a sequence of integers satisfying $\limsup_{N \to \infty} k(N)/N^d < p_0$. Consider an MFP model with parameters $p_0$ and $N$. For all $\epsilon>0$ there exists $N_0$ such that $\IP_{p_{0}}(E) > 1-\epsilon$ for all $N \geq N_0$.
\end{lemma}

\begin{proof} Let $\delta>0$ and $N_0$ be such that $k(N)/N^d \leq p_0-2\delta =:p$ for all $N \geq N_0$. Choose $N_1 \geq N_0$ so large that $p_0/(4\delta^2 N^d) < \delta $ for $N \geq N_1$. We will show that
\begin{equation}\label{induction E}
\IP_{p_{0}}(E_m) \geq 1-\frac{1}{4 \delta^2 N^{d}} ,
\end{equation}
for all $m\geq 0$ and $N\geq N_1$. Since $E_m$ decreases to $E$ as $m\to\infty$, it follows that 
\[
 \IP_{p_{0}}(E) = \lim_{m \to \infty} \IP_{p_{0}}(E_m) \geq 1-\frac{1}{4 \delta^2 N^{d}},
\]
for $N \geq N_1$. Now take $N_2\geq N_1$ so large that $1- \frac{1}{4 \delta^2 N^d} > 1- \epsilon$ for all $N \geq N_2$. It remains to show \eqref{induction E}.

We prove \eqref{induction E} by induction on $m$. Consider the event $E_0$, i.e.\ the event that the unit cube contains at least $k(N)$ retained level-1 subcubes. Let $X(n,p)$ denote a binomially distributed random variable with parameters $n\in \IN$ and $p\in [0,1]$. Since the number of retained level-1 cubes has a binomial distribution with parameters $N^{d}$ and $p_0$, it follows from Chebyshev's inequality that, for every $N \geq N_1$, we have (writing $\IP$ for the probability measure governing the binomially distributed random variables)
\begin{eqnarray*}
 \IP_{p_{0}}(E_0) &=& \IP(X(N^d,p_0) \geq k(N))  \\
 &\geq& \IP(X(N^d,p_0) \geq p N^d) \\
&\geq&  1- \frac{\Var X(N^d,p_0)}{4\delta^2 N^{2d}} \\
&=& 1 -\frac{p_0(1-p_0)N^d}{4 \delta^2 N^{2d}} \\
&\geq& 1- \frac{1}{4 \delta^2 N^{d}}.
\end{eqnarray*}

Next, let $m\geq 0$ and $N \geq N_1$ and suppose that \eqref{induction E} holds for this $m$ and $N$. Recall that $E_{m+1}$ is the event that the unit cube contains at least $k(N)$ $m$-good level-1 cubes. The probability that a level-1 cube is $m$-good, given that it is retained, is equal to $\IP_{p_0}(E_m)$. Using the induction hypothesis, we get
\begin{eqnarray*}
 \IP_{p_{0}}(E_{m+1}) &=& \IP(X(N^{d},p_{0}\IP_{p_{0}}(E_m)) \geq k(N)) \\
& \geq& \IP(X(N^d,p_0(1-\textstyle{\frac{1}{4 \delta^2 N^{d}}})) \geq k(N)). 
\end{eqnarray*}
By our choices for $\delta$ and $N$ it follows that $p_0(1-\frac{1}{4 \delta^2 N^{d}}) > p+\delta$. Hence, using Chebyshev's inequality, we get
\begin{eqnarray*}
\IP(X(N^d,p_0(1-\textstyle{\frac{1}{4 \delta^2 N^{d}}})) \geq k(N))
&\geq&\IP(X(N^d, p+\delta) \geq k(N)) \\
&\geq& \IP(X(N^d, p+ \delta) \geq pN^d) \\
&\geq& 1- \frac{\Var X(N^d,p+\delta)}{\delta^2 N^{2d}} \\
 &\geq& 1-  \frac{1}{4\delta^2N^{d}}.
\end{eqnarray*}
Therefore, the induction step is valid and we have proved \eqref{induction E}.
\end{proof}

\begin{proof}[Proof of Theorem \ref{thm subcritical}]
Let $p,p_0$ be such that $p < p_0 < \pc$. Let $(k(N))_{N \geq 2}$ be a sequence such that $k(N)/N^d \leq p$, for all $N \geq 2$, and $k(N)/N^d \to p$ as $N \to \infty$. Consider an MFP model with parameters $p_0$ and $N$ and define the event $E$ as in \eqref{E}. Henceforth, we assume that $N$ is so large that $\IP_{p_0}(E)>0$, which is possible by Lemma \ref{lem high prob}. In order to prove Theorem \ref{thm subcritical} we will use $E$ to construct two random subsets, $\tilde{D}_{p_{0}}$ and $\tilde{D}_{k(N)}$, of the unit cube, on a common probability space and with the following properties: 
\begin{itemize}
 \item[(i)] $\tilde{D}_{k(N)} \subset \tilde{D}_{p_{0}}$;
\item[(ii)] the law of $\tilde{D}_{p_{0}}$ is stochastically dominated by the conditional law of $D_{p_{0}}$, conditioned on the event $E$;
\item[(iii)] the law of $\tilde D_{k(N)}$ is the same as the law of $D_{k(N)}$.
\end{itemize}

It follows that the law of $D_{k(N)}$ is stochastically dominated by the conditional law of $D_{p_0}$, 
conditioned on the event $E$. Hence, the probability that the unit cube is crossed by $D_{k(N)}$ is at most the conditional probability that $D_{p_0}$ crosses the unit cube, conditioned on the event $E$. By \eqref{pc subcritical} the MFP process with parameter $p_0$ is subcritical, thus the latter probability equals zero. Using the fact that $k(N)/N^d \to p$ as $N \to \infty$, we conclude that
\[
 \liminf_{N \to \infty} \frac{k_{c}(N,d)}{N^{d}} \geq p.
\]
Since $p<\pc$ was arbitrary, we get
\[
 \liminf_{N \to \infty} \frac{k_{c}(N,d)}{N^{d}} \geq \pc.
\]

It remains to construct random sets $\tilde{D}_{p_{0}},\tilde{D}_{k(N)}$ with the properties (i)-(iii). First we construct two sequences $(\tilde{D}_{p_{0}}^n)_{n\geq 1}, (\tilde{D}_{k(N)}^n)_{n\geq 1}$ of decreasing random subsets.
Let ${\cal L}$ be the conditional law of the number of $\infty$-good level-1 cubes of the MFP process, conditioned on the event $E$. Note that the support of ${\cal L}$ is $\{k(N),k(N)+1,\ldots,N^d\}$. Furthermore, for a fixed level-$n$ cube $C(\Ib)$, ${\cal L}$ is also equal to the conditional law of the number of $\infty$-good level-$(n+1)$ subcubes in $C(\Ib)$, conditioned on $C(\Ib)$ being $\infty$-good. 

Choose an integer $l$ according to ${\cal L}$ and choose $l$ level-1 cubes uniformly. Define $\tilde D_{p_0}^1$ as the closure of the union of these $l$ level-1 cubes. Choose $k(N)$ out of these $l$ cubes in a uniform way and define $\tilde D_{k(N)}^1$ as the closure of the union of these $k(N)$ cubes.  For each level-1 cube $C({\bf I}) \subset \tilde D_{p_0}^1$, pick an integer $l({\bf I})$ according to ${\cal L}$, independently of other cubes, and choose $l({\bf I})$ level-2 subcubes of $C({\bf I})$ in a uniform way. Define $\tilde D_{p_0}^2$ as the closure of the union of all selected level-2 cubes. For each level-1 cube $C({\bf I}) \subset \tilde D_{k(N)}^1$, uniformly choose $k(N)$ out of the $l({\bf I})$ selected level-2 subcubes. Define $\tilde D_{k(N)}^2$ as the closure of the union of the $k(N)^2$ selected level-2 cubes of $C(\Ib)$. Iterating this procedure yields two infinite decreasing sequences of random subsets $(\tilde{D}_{p_{0}}^n)_{n\geq 1}, (\tilde{D}_{k(N)}^n)_{n\geq 1}$.

Now define 
\[
 \tilde D_{p_0} := \bigcap_{n=1}^{\infty} \tilde D_{p_0}^n,\quad \tilde D_{k(N)} := \bigcap_{n=1}^{\infty} \tilde D_{k(N)}^n.
\]
By construction, for each $n\geq 1$, we have that (1) $\tilde D_{k(N)}^n \subset \tilde D_{p_0}^n$, (2) the law of $\tilde{D}_{p_{0}}^n$ is stochastically dominated by the conditional law of $D_{p_{0}}^n$ given $E$ and (3) the law of $\tilde D_{k(N)}^n$ is equal to the law of $D_{k(N)}^n$. It follows that the limit sets $\tilde{D}_{p_{0}},\tilde{D}_{k(N)}$ satisfy properties (i)-(iii).
\end{proof}

\subsection{Proof of Theorem \ref{thm supercritical}}\label{sec supercritical}

Let us start by outlining the proof. The first part consists mainly of setting up the framework, where we use the notation of Falconer and Grimmett \cite{FalGri}, which will enable us in the second part to prove that the subcubes of the fractal process satisfy certain ``good'' properties with probability arbitrarily close to 1 as $N\to \infty$. Informally, a subcube is good when there exist many connections inside the cube between its faces and when it is also connected to other good subcubes. Therefore, the probability of crossing the unit cube converges to 1 as $N\to\infty$.

Although we will partly follow \cite{FalGri}, it does not seem possible to use Theorem 2.2 of \cite{FalGri} directly. 
First, we state (a slightly adapted version of) Lemma 2 of \cite{FalGri}, which concerns site percolation with parameter $\pi$ on $\IL^d$. We let every vertex of $\IL^d$ be colored \emph{black} with probability $\pi$ and \emph{white} otherwise, independently of other vertices. We write $P_\pi$ for the ensuing product measure with density $\pi \in [0,1]$. We call a subset $C$ of $\IL^{d}$ a \emph{black cluster} if it is a maximal connected subset (with respect to the adjacency relation on $\IL^d$) of black vertices. Denote the cube with vertex set  $\{1,2,\ldots,N\}^d$ by $B_N$. Let $\LL$ be the set of edges of the unit cube $[0,1]^d$, that is $\LL$ contains all sets of the form
\[
 L_{r}(\mathbf{a})=\{a_1\} \times \{a_2\} \times \cdots \times \{a_{r-1}\} \times [0,1] \times \{a_{r+1}\} \times \cdots \times \{a_d\}
\]
as $r$ ranges over $\{1,\ldots,d\}$ and $\mathbf{a} = (a_1,a_2,\ldots,a_d)$ ranges over $\{0,1\}^d$. For each $L=L_{r}(\mathbf{a}) \in \LL$ we write 
\[
L_N=\{ \mathbf{x} \in B_N: x_{i} = \max\{1,a_{i}N\} \text{ for } 1 \leq i \leq d, i \neq r \}
\]
for the corresponding edge of $B_N$.

\begin{lemma}\label{lem FalGri}
 Suppose $\pi>\pc, \epsilon>0$ and let $q$ be a positive integer. There exist positive integers $u$ and $N_1$ such that the following holds for all $N \geq N_1$. Let $U(1),\ldots,U(q)$ be subsets of vertices of $B_N$ such that for each $r\in\{1,\ldots,q\}$, (i) $|U(r)| \geq u$ and (ii) there exists $L \in \LL$ such that $U(r) \subset L_N$. Then, 
\begin{equation}\label{perc bound}
 P_{\pi}\left(\begin{array}{c}
      \text{there exists a black cluster } C_N \text{ such that }|C_N \cap L_N| \geq u \\
      \text{ for all } L \in \LL, \text{ and } | C_N \cap U(r)| \geq 1, \text{ for all } r\in\{1,\ldots,q\} \end{array}\right) \geq 1- \frac{\epsilon}{2}.
\end{equation}
\end{lemma}

Our goal is to show that the following holds uniformly in $n$: With probability arbitrarily close to 1 as $N \to \infty$, there is a sequence of cubes in $D_{k(N)}^n$, each with at least one edge in common with the next, which crosses the unit cube. In order to prove this we examine the cubes $C(\Ib)$, for $\Ib\in\JJ^{(n)}$, in turn according to the ordering on $\JJ^{(n)}$, and declare some of them to be good according to the rule given below. Since the probabilistic bounds on the goodness of cubes will hold uniformly in $n$, the desired conclusion follows.

Fix integers $n,u,k \geq 1$ until Lemma \ref{lem prob}. For $m\geq 1$, identify a level-$m$ cube with a vertex in $B_{N^m} \subset \IL^d$ in the canonical way. A set of level-$m$ cubes $\{C(\Ib_1),\ldots,C(\Ib_l)\}$ is called \emph{edge-connected} if they form a connected set with respect to the adjacency relation of $\IL^d$. Whether a cube $C(\Ib)$, for $\Ib \in \JJ^{(n)}$, is called \emph{$(n,u)$-good} or not, is determined by the following inductive procedure. Let $\Ib \in \JJ^{(n)}$, and assume that the goodness of $C(\Ib')$ has been decided for all $\Ib' < \Ib$. We have the following possibilities:
\begin{itemize}
 \item[(a)] $|\Ib|=n$. Then $C(\Ib)$ is always declared $(n,u)$-good.
\item[(b)] $0 \leq |\Ib| = m<n$.
\end{itemize}
In the latter case we act as follows. Note that the subcubes $C(\Ib,\mathbf{j})$ with $\Ij \in J^d$ have already been examined, since $(\Ib,\Ij) < \Ib$.
Define the following set of level-$(m+1)$ subcubes of $C(\Ib)$, 
\begin{equation}\label{good subcubes}
   \DD(\Ib):= \{ C(\Ib,\Ij): \Ij \in J^d \text{ with } C(\Ib,\Ij) \, (n,u)\text{-good and } Z_k(\Ib,\Ij) =1 \}. 
 \end{equation}
We declare $C(\Ib)$ to be $(n,u)$-good if there exists an edge-connected set $\HH(\Ib)\subset \DD(\Ib)$ such that
\begin{itemize}
 \item[(i)] Each edge of $C(\Ib)$ intersects at least $u$ cubes of $\HH(\Ib)$;
\item[(ii)] For every $(n,u)$-good level-$m$ cube $C(\Ib')$ with $\Ib'<\Ib$ that has (at least) one edge in common with $C(\Ib)$, there are a cube of $\HH(\Ib')$ and a cube of $\HH(\Ib)$ with a common edge.
\end{itemize}
(If there is more than one candidate for $\HH(\Ib)$ we use some deterministic rule to choose one of them.)
This procedure determines whether $C(\Ib)$ is $(n,u)$-good for each $\Ib$ in turn. Note that it is easier for higher level cubes to be $(n,u)$-good than for lower level cubes. In particular, for the unit cube, i.e.\ $C(\emptyset)$, it is the hardest to be $(n,u)$-good.

The next lemma shows that if the unit cube is $(n,u)$-good then there is a sequence of cubes in $D_k^n$, each with at least one edge in common with the next, which connects the ``left-hand side'' of $[0,1]^d$ with its ``right-hand side''. If such a sequence of cubes exists in $D_k^n$ we say that \emph{percolation occurs in $D_{k}^{n}$}.

\begin{lemma}\label{lem good}
 Suppose $[0,1]^d$ is $(n,u)$-good, then percolation occurs in $D_{k}^{n}$.
\end{lemma}

\begin{proof}
Assume that the unit cube, i.e.\ $C(\emptyset)$, is $(n,u)$-good. We will show, with a recursive argument, that for $1\leq m \leq n$ there exists an edge-connected chain of retained $(n,u)$-good level-$m$ cubes which joins $\{0\} \times [0,1]^{d-1}$ and $\{1\} \times [0,1]^{d-1}$. In particular, this holds for $m=n$ and hence percolation occurs in $D_k^n$.

Since the unit cube is assumed to be $(n,u)$-good, $\DD(\emptyset)$ contains by definition an edge-connected subset $\HH(\emptyset)$ of retained $(n,u)$-good level-1 subcubes, such that each edge of $C(\emptyset)$ intersects at least $u$ cubes of $\HH(\emptyset)$. In particular, there is a sequence of retained $(n,u)$-good edge-connected level-1 cubes that connects the left-hand side of $[0,1]^d$ with its right-hand side.

Let $1\leq m < n$ and assume that there exists an edge-connected chain $C(\Ib_1),\ldots,C(\Ib_l)$ of retained $(n,u)$-good level-$m$ cubes which connects the left-hand side of $[0,1]^d$ with its right-hand side. For each $i$, $1 \leq i \leq l$, either $\Ib_i < \Ib_{i+1}$ or $\Ib_{i+1}<\Ib_i$. By condition (ii), there exist level-$(m+1)$ cubes of $\HH(\Ib_{i+1})$ which are edge-connected to level-$(m+1)$ cubes of $\HH(\Ib_i)$. These level-$(m+1)$ cubes $C(\Jb)$ are all $(n,u)$-good and have $Z_k(\Jb)=1$, by \eqref{good subcubes} and the definition of the $\HH(\Ib)$. It follows that there is an edge-connected chain of retained $(n,u)$-good level-$(m+1)$ cubes $C(\Jb)$ which joins $\{0\} \times [0,1]^{d-1}$ and $\{1\} \times [0,1]^{d-1}$. 
\end{proof}

For $\Ib \in \JJ^{(n)}$, define the index $\Ib^- \in \JJ^{(n)}$ by
\[
 \Ib^- = \max\{ \Ib': \Ib' < \Ib \text{ and } |\Ib'| \leq |\Ib| \}.
\]
If there is no such index, $\Ib^-$ is left undefined. For each $\Ib \in \JJ^{(n)}$ we let $\FF(\Ib)$ denote the $\sigma$-field 
\[
 \FF(\Ib) = \sigma(Z_k(\Ib',\Ij): |\Ib'| \leq n-1, \Ib' \leq \Ib, \Ij \in J^{d}).
\]
If $\Ib^-$ is undefined, we take $\FF(\Ib^-)$ to be the trivial $\sigma$-field. Note that $\FF(\Ib)$ is generated by those $Z_k$ that have been examined prior to deciding whether $C(\Ib)$ is $(n,u)$-good. In particular, by virtue of the ordering on the cubes as introduced in Section \ref{sec pre}, $\FF(\Ib^-)$ does \emph{not} contain any information about subcubes of $\Ib$.

Let $p>\pc$ and let $(k(N))_{N \geq 2}$ be a sequence such that $k(N)/N^d \geq p$, for all $N \geq 2$. We want to prove that, for every $\epsilon>0$, the probability that $[0,1]^d$ is $(n,u)$-good in the $k(N)$-model is at least $1-\epsilon$, 
for $N \geq N_0$, where $N_0$ is an integer which has to be taken sufficiently large to satisfy certain probabilistic bounds but is independent of $n$.

Let us first give a sketch of the proof. Fix $N\geq N_0$ and consider the $k(N)$-model. We use a recursive argument. The smallest level-$n$ cube according to the ordering on $\JJ^{(n)}$ is by definition $(n,u)$-good. Let $\Ib \in \JJ^{(n)}$ and assume that $\IP_{k(N)}(C(\Ib') \text{ is } (n,u)\text{-good} \mid \FF(\Ib'^-)) \geq 1-\epsilon$ for all $\Ib'<\Ib$. We prove that, given $\FF(\Ib^-)$, $C(\Ib)$ is $(n,u)$-good with probability at least $1-\epsilon$. The proof of this consists of a coupling between a product measure with density $\pi \in (\pc,(1-\epsilon)p)$ in the box $B_N$ and the law of the set of subcubes $C(\Ib,\Ij)$ of $C(\Ib)$ which are $(n,u)$-good and satisfy $Z_{k(N)}(\Ib,\Ij) = 1$. Applying Lemma \ref{lem FalGri} to the product measure combined with the coupling yields that the subcubes satisfy properties (i) and (ii) with probability at least $1-\epsilon$. Therefore, given $\FF(\Ib^-)$, $C(\Ib)$ is $(n,u)$-good with probability at least $1-\epsilon$. Iterating this argument then yields that the unit cube is $(n,u)$-good with probability at least $1-\epsilon$, for $N \geq N_0$.

The proof in \cite{FalGri} of the analogous result that $\sigma(p,N,d) \to 1$ as $N \to \infty$ for $p> \pc$ is considerably less involved. In the context of \cite{FalGri}, subcubes are retained with probability $p$ independently of other cubes, which is not the case in $k$-fractal percolation. Therefore, they can directly show that there exists $\pi>\pc$ such that, for $\Ib \in \JJ^{(n)}$, the law of the set of subcubes $C(\Ib,\Ij)$ of $C(\Ib)$ which are good and satisfy $Z_p(\Ib,\Ij) = 1$, dominates an i.i.d.\ process on the box $B_N$ with density $\pi$.

We need the following result for binomially distributed random variables, which we state as a lemma for future reference. Since the result follows easily from Chebyshev's inequality, we omit the proof.

\begin{lemma}\label{lem good3}
Let $p > \pc$ and let $(k(N))_{N \geq 2}$ be a sequence of integers such that $k(N)/N^d \geq p$ for all  $N \geq 2$. Let $\epsilon>0$ be such that $(1-\epsilon) p > \pc$, let $\pi \in (\pc, (1-\epsilon) p)$ and define $M:= ((1-\epsilon)p+\pi)N^d / 2$. There exists $N_2$ such that 
\[
\IP(\{X(k(N),1-\epsilon) \geq M\} \cap \{X'(N^{d},\pi) \leq M\}) \geq 1-\epsilon/2,
\]
for $N \geq N_2$, where $X$ and $X'$ are independent, binomially distributed random variables with the indicated parameters.
\end{lemma}

We now prove that, for any $\epsilon>0$, the unit cube is $(n,u)$-good with probability at least $1-\epsilon$, for $N$ large enough but independent of $n$.

\begin{lemma} \label{lem prob}
Let $p> \pc$ and let $(k(N))_{N \geq 2}$ be a sequence of integers such that $k(N)/N^{d} \geq  p$, for all $N \geq 2$. Let $\epsilon>0$ be such that $(1-\epsilon)p>\pc$. Take $\pi \in (\pc, (1-\epsilon)p)$ and set $q=3^d$. Let $u$ and $N_1$ be given by Lemma \ref{lem FalGri}. Let $N_2$ be given by Lemma \ref{lem good3}.  Set $N_0 =\max\{N_1,N_2\}$. Then, for all $n \geq 1$,
\begin{equation}\label{unit cube good}
 \IP_{k(N)}([0,1]^d \text{ is } (n,u)\text{-good}) \geq 1-\epsilon,
\end{equation}
for all $N \geq N_0$.
\end{lemma}

\begin{proof}
Fix $N \geq N_0$ and $n\geq 1$ and consider the $k(N)$-fractal model. Our aim is to show that 
\begin{equation}\label{induction2}
 \IP_{k(N)}(C(\Ib) \text{ is } (n,u)\text{-good} \mid \FF(\Ib^-)) \geq 1 -\epsilon
\end{equation}
holds for all $\Ib \in \JJ^{(n)}$. Taking $\Ib=\emptyset$ then yields \eqref{unit cube good}. We prove this with a recursive argument. Let $\Ib_0$ be the smallest index in $\JJ^{(n)}$, according to the ordering on $\JJ^{(n)}$. By virtue of the ordering, we have $|\Ib_0| = n$. Hence, by definition, $C(\Ib_0)$ is $(n,u)$-good. In particular, \eqref{induction2} holds for $\Ib_0$.

The recursive step is as follows. Take an index $\Ib \in \JJ^{(n)}$ and assume that  
\begin{equation}\label{induction}
 \IP_{k(N)}(C(\Ib') \text{ is } (n,u)\text{-good} \mid \FF(\Ib'^-)) \geq 1 -\epsilon,
\end{equation}
has been established for all indices $\Ib'$ in $\JJ^{(n)}$ less than $\Ib$. We have to show that \eqref{induction2} holds for $\Ib$ given this assumption. We have two cases:
\begin{itemize}
 \item[(a)] $|\Ib| =n$; then $\IP_{k(N)}(C(\Ib) \text{ is } (n,u)\text{-good}) =1$ and \eqref{induction2} is true.
\item[(b)] $0 \leq |\Ib| = m < n$.
\end{itemize}

For case (b), given $\FF(\Ib^-)$, the goodness of $C(\Ib')$ is determined (in particular) for all $\Ib' < \Ib$ with $|\Ib|=m$. Let
\begin{eqnarray*}
 \QQ=\left\{\begin{array}{cc} \Ib': & \Ib'<\Ib \text{ and } C(\Ib') \text{ is an } (n,u)\text{-good level-}m \\
& \text{ cube with an edge in common with }C(\Ib)           
          \end{array}\right\}.
\end{eqnarray*}
For each $\Ib' \in \QQ$, let $E(\Ib')$ be some common edge of $C(\Ib)$ and $C(\Ib')$. Since $C(\Ib')$ is $(n,u)$-good, there are at least $u$ level-$(m+1)$ subcubes in $\HH(\Ib')$ which intersect $E(\Ib')$; call this set of subcubes $\UU(\Ib')$. To see whether $C(\Ib)$ is $(n,u)$-good, we look at $C(\Ib,\Ij(l))$ where $\Ij(l), 1 \leq l \leq N^d,$ are the vectors of $J^d$ arranged in order. We have $(\Ib,\Ij(l)) < \Ib$, so by the induction hypothesis \eqref{induction} we have
\begin{equation}\label{induction sub}
 \IP_{k(N)}(C(\Ib,\Ij(l)) \text{ is } (n,u)\text{-good} \mid \FF((\Ib,\Ij(l))^-)) \geq 1-\epsilon,
\end{equation}
for all $l$. Note that $\FF((\Ib,\Ij(1))^-) = \FF(\Ib^-)$.

We identify each subcube of $C(\Ib)$ in the canonical way with a vertex in $B_N$. We will construct three random subsets $G_1, G_2, G_3$ of $B_N$ on a common probability space with the following properties:
\begin{itemize}
\item[(I)] the law of $G_1$ equals the law of the set of subcubes $C(\Ib,\Ij)$ of $C(\Ib)$ which are $(n,u)$-good and satisfy $Z_{k(N)}(\Ib,\Ij)=1$;
\item[(II)] $G_2$ is obtained by first selecting $k(N)$ vertices of $B_N$ uniformly and then retaining each selected vertex with probability $1-\epsilon$, independently of other vertices;
\item[(III)] the law of $G_3$ is the Bernoulli product measure with density $\pi$ on $B_N$;
\item[(IV)] $G_1 \supset G_2$;
\item[(V)] $\IP(G_2 \supset G_3) \geq 1-\epsilon/2$.
\end{itemize}

From \eqref{induction sub} and a standard coupling technique, sometimes referred to as sequential coupling (see e.g.\ \cite{LigSte}), the construction of $G_1$ and $G_2$ with properties (I), (II) and (IV) is straightforward. The construction of $G_3$ such that properties (III) and (V) hold is given below. Let $|G_2|$ denote the cardinality of the set $G_2$. Define $M=((1-\epsilon)p+\pi)N^d / 2$ and let $R$ be a number drawn from a binomial distribution with parameters $N^d$ and $\pi$, independently of $G_1$ and $G_2$. If $|G_2| \geq M$ and $M\geq R$ we select $R$ vertices uniformly  out of the $|G_2|$ retained vertices of $G_2$ and call this set $G_3$. Otherwise, we select, independently of $G_1$ and $G_2$, $R$ vertices of $B_N$ in a uniform way and call this set $G_3$. From the construction (note that also $G_2$ was obtained in a uniform way) it is clear that $G_3$ satisfies property (III). Observe that $|G_2|$ has a binomial distribution with parameters $k(N)$ and $1-\epsilon$. From Lemma \ref{lem good3} it follows that
\[
 \IP(\{|G_2| \geq M \} \cap \{R \leq M\}) \geq 1-\epsilon/2.
\]
Hence, property (V) also holds.

Let us now return to the goodness of $C(\Ib)$. As before, we identify the random subsets $G_1,G_2,G_3$ of $B_N$ with
the corresponding sets of subcubes  of $C(\Ib)$ in the canonical way. It then follows from property (III) and Lemma \ref{lem FalGri} (note that $\QQ$ has cardinality at most $3^d =q$) that $G_3$ has an edge-connected subset which satisfies the following properties with probability at least $1-\epsilon/2$:
\begin{itemize}
  \item[(i)] intersects every edge of $C(\Ib)$ with at least $u$ cubes;
 \item[(ii)] contains a cube that is edge-connected to a cube of $\UU(\Ib')$, for all $\Ib' \in \QQ$. 
 \end{itemize}

Combining properties (IV), (V) and the previous paragraph we obtain
\begin{eqnarray*}
 \lefteqn{\IP_{k(N)}(C(\Ib) \text{ is } (n,u)\text{-good} \mid \FF(\Ib^-))} \\
& \geq& \IP(\{G_1 \supset G_3 \} \cap \{G_3 \text{ satisfies properties (i) and (ii)}\}) \\
&\geq& 1-\epsilon.
\end{eqnarray*}
Therefore, \eqref{induction2} holds for the index $\Ib$ given that \eqref{induction} holds for all indices $\Ib'<\Ib$. A recursive use of this argument -- recall that \eqref{induction2} is valid for $\Ib_0$ (the smallest index according to the ordering) -- yields that \eqref{induction2} holds for all $\Ib$. Taking $\Ib=\emptyset$ in \eqref{induction2} proves the lemma.
\end{proof}

We are now able to conclude the proof of Theorem \ref{thm supercritical}.

\begin{proof}[Proof of Theorem \ref{thm supercritical}]
Let $p>\pc$ and consider a sequence $(k(N))_{N \geq 2}$ such that $k(N)/N^d \geq p$, for all $N \geq 2$. We get, using both Lemma \ref{lem prob} and Lemma \ref{lem good}, that for any $\epsilon >0$ such that $(1-\epsilon)p > \pc$, there exists $N_0$, depending on $\epsilon$, such that
\begin{equation}\label{level percolation}
 \IP_{k(N)}(\text{percolation in } D^n_{k(N)}) \geq \IP_{k(N)}([0,1]^d \text{ is } (n,u)\text{-good}) \geq 1-\epsilon,
\end{equation}
for $N \geq N_0$. It is well known (see e.g.\ \cite{FalGri}) that
\[
 \{[0,1]^d \text{ is crossed by } D_{k(N)} \} = \bigcap_{n=1}^{\infty} \{\text{percolation in } D^n_{k(N)}\}. 
\]
Hence, taking the limit $n\to \infty$ in \eqref{level percolation} yields that for $\epsilon > 0$ small enough
\begin{equation}\label{final percolation}
 \IP_{k(N)}([0,1]^d \text{ is crossed by } D_{k(N)}) \geq 1-\epsilon,
\end{equation}
for $N \geq N_0$. Therefore, 
\[
 \theta(k(N),N,d) \to 1,
\]
as $N\to \infty$.
\end{proof}

\subsection{Proof of Theorem \ref{thm generalization}}\label{sec generalization}

\begin{proof}[Proof of Theorem \ref{thm generalization}]
We use the idea of the proof of Theorem \ref{thm subcritical} and the result of Theorem \ref{thm supercritical}. Fix some $p_0$ such that $\pc<p_0 < p$ and set $k(N):=\lfloor p_0 N^d \rfloor$. Consider the event $F$ that in the GFP model with generator $Y = Y(N,d)$ there exists an infinite tree of retained subcubes such that each subcube in the tree contains at least $k(N)$ retained subcubes in the tree. Similar to the proof of Lemma \ref{lem high prob}, we prove that $\IP(F)\to 1$ as $N\to\infty$. We then show that the law of $D_{k(N)}$ is stochastically dominated by the conditional law of $D_Y$, conditioned on the event $F$. By Theorem \ref{thm supercritical} we can then conclude that $\phi(Y(N,d),N,d) \to 1$ as $N \to \infty$.

Consider the construction of $D_{Y}$. We will use the same definition of $m$-good as in Section \ref{sec subcritical}, that is, if a level-$n$ cube is retained and contains at least $k(N)$ retained subcubes, we call this level-$n$ cube 0-good. Recursively, we say that a level-$n$ cube is $(m+1)$-good if it is retained and contains at least $k(N)$ $m$-good level-$(n+1)$ subcubes. We call the unit cube $\infty$-good if it is $m$-good for every $m \geq 0$. Define the following events
\begin{eqnarray*}
 F_m:=\{[0,1]^d \text{ is }m\text{-good}\},\\
F:=\{[0,1]^d \text{ is }\infty\text{-good}\}.
\end{eqnarray*}

We will show that for every $\epsilon>0$ such that $(1-\epsilon)p > p_0$ there exists $N_0=N_0(\epsilon)$ such that, for all $m \geq 0$,
\begin{equation}\label{claim m-good}
\IP(F_m) > 1-\epsilon, \quad \text{for all } N \geq N_0 .
\end{equation}

The proof of \eqref{claim m-good} is similar to the proof of Lemma \ref{lem high prob}. Let $\epsilon>0$ be such that $(1-\epsilon)p > p_0$. Take $\delta>0$ such that $(1-\epsilon)p>p_0+\delta$. Then, take $N_0$ so large that
\begin{eqnarray}
 1 - \frac{1}{4\delta^2 N} > 1-\epsilon/2 \quad \text{and} \label{delta bound}\\
\IP(Y \geq p N^d) > 1-\epsilon/2 \label{Y bound},
\end{eqnarray}
for all $N\geq N_0$. We prove that \eqref{claim m-good} holds for this $N_0$ and all $m\geq 0$, by induction on $m$. Since $k(N) = \lfloor p_0 N^d \rfloor \leq pN^d$ it follows from \eqref{Y bound} that $\IP(F_0)>1-\epsilon$, for all $N\geq N_0$.

Next, assume that \eqref{claim m-good} holds for some $m\geq 0$. The probability that a level-1 cube is $m$-good, given that it is retained, is equal to $\IP(F_m)$. It follows that, given that the number of retained level-1 cubes equals $y$, the number of $m$-good level-1 cubes has a binomial distribution with parameters $y$ and $\IP(F_m)$. By our choices for $N_0$ and $\delta$ we get
\begin{eqnarray*}
 \IP(F_{m+1}) &=& \sum_{y \geq k(N)} \IP(X(y,\IP(F_m)) \geq k(N)) \, \IP(Y = y) \\
& \geq& \IP(X(\lfloor pN^d \rfloor,\IP(F_m)) \geq p_0 N^d) \, \IP(Y \geq \lfloor pN^d \rfloor) \\
&\geq& \IP(X(\lfloor pN^d \rfloor,1-\epsilon)\geq p_0 N^d) (1-\epsilon/2) \\
&\geq& \left(1-\frac{\Var X(\lfloor pN^d \rfloor,1-\epsilon)}{(p_0 - (1-\epsilon)p)^2 N^{2d}}\right)(1-\epsilon/2) \\
&\geq& \left(1-\frac{(1-\epsilon)\epsilon pN^d}{\delta^2 N^{2d}}\right) (1-\epsilon/2) \\
&\geq& \left(1-\frac{1}{4 \delta^2 N^d}\right)(1-\epsilon/2) \\
&\geq& (1-\epsilon/2)(1-\epsilon/2)>1-\epsilon,
\end{eqnarray*}
 for all $N \geq N_0$. Hence, the induction step is valid.

Analogously to the proof of Theorem \ref{thm subcritical} we use the event $F=\bigcap_{m=1}^{\infty} F_m$ to construct two random subsets $\tilde D_{k(N)}$ and $\tilde D_{Y}$ on a common probability space, with the following properties: 
\begin{itemize}
 \item[(i)] $\tilde{D}_{k(N)} \subset \tilde{D}_{Y}$;
\item[(ii)] the law of $\tilde{D}_{Y}$ is stochastically dominated by the conditional law of $D_{Y}$, conditioned on the event $F$;
\item[(iii)] the law of $\tilde D_{k(N)}$ is equal to the law of $D_{k(N)}$.
\end{itemize}
This construction is the same (modulo replacing the binomial distribution with $Y$) as in the proof of Theorem \ref{thm subcritical} and is therefore omitted.

From properties (i)-(iii) and Theorem \ref{thm supercritical} we get
\begin{eqnarray*}
 \lefteqn{\IP([0,1]^d \text{ is crossed by } D_{Y(N,d)} | F) } \\
&\geq& \IP([0,1]^d \text{ is crossed by } \tilde D_{Y(N,d)})  \\
&\geq & \IP([0,1]^d \text{ is crossed by } \tilde D_{k(N)}) \\
&=&\IP([0,1]^d \text{ is crossed by } D_{k(N)}) \to 1,
\end{eqnarray*}
as $N \to \infty$. Since \eqref{claim m-good} implies that $\IP(F) \to 1$ as $N \to \infty$, we obtain
\[
 \IP([0,1]^d \text{ is crossed by } D_{Y(N,d)}) \to 1,
\]
as $N \to \infty$.
\end{proof}

\section{Proofs of the fat fractal results} \label{sec main fat}

In this section we prove our results concerning fat fractal percolation. First, we state an elementary property of the fat fractal percolation model; it follows immediately from Fubini's theorem and we omit the proof.

\begin{proposition}\label{prop expectation}
The expected Lebesgue measure of the limit set of fat fractal percolation is given by 
\[
\IE \lambda(\Dfat) = \prod_{n=1}^{\infty} p_n.
\]
\end{proposition}

\subsection{Proof of Theorem \ref{thm12}}

Since $\prod_{n=1}^{\infty} p_n >0$ it follows from Proposition \ref{prop expectation} that with positive probability the limit set has positive Lebesgue measure given $\Dfat \neq \emptyset$. Theorem \ref{thm12} states that the latter holds with probability 1.

\begin{proof}[Proof of Theorem \ref{thm12}]
Let $Z_n$ denote the number of retained level-$n$ cubes after iteration step $n$ and set $Z_0:=1$. Since the retention probabilities $p_n$ vary with $n$, the process $(Z_n)_{n \geq 1}$ is a so-called branching process in a time-varying environment. Following the notation of Lyons in \cite{Lyo} let $L_n$ be a random variable, having the distribution of $Z_n$ given that $Z_{n-1}=1$. Note that $L_n$ has a binomial distribution with parameters $N^{d}$ and $p_{n}$. 

Define the process $(W_n)_{n \geq 1}$ by
\[
 W_n := \frac{Z_n}{\prod_{i=1}^{n} p_i N^d}.                                                                                                                                                                                                                                                                                                                                             \]
It is straightforward to show that $(W_n)_{n \geq 1}$ is  a martingale:
\begin{eqnarray*}
 \IE[W_n | W_{n-1}] &=&\frac{\IE[Z_n | Z_{n-1}]}{\prod_{i=1}^{n} p_i N^d} = \frac{Z_{n-1}}{\prod_{i=1}^{n} p_i N^d} \IE[Z_n | Z_{n-1}=1] \\
&=& \frac{Z_{n-1} p_n N^d}{\prod_{i=1}^{n} p_i N^d} = W_{n-1}.
\end{eqnarray*}
The Martingale Convergence Theorem tells us that $W_n$ converges almost surely to a random variable $W$. Theorem 4.14 of \cite{Lyo} states that if
\[
 A:=\sup_{n} ||L_n ||_{\infty} < \infty,
\]
then $W>0$ a.s.\ given non-extinction. It is clearly the case that $A <\infty$, because $L_n$ can take at most the value $N^{d}$. Therefore, $W_n$ converges to a random variable $W$ which is stricly positive a.s.\ given non-extinction.

The Lebesgue measure of the retained cubes at each iteration step $n$ is equal to $Z_n / N^{dn}$. We have
\begin{equation}\label{eqn leb}
 \lambda(\Dfat^n) = \frac{Z_n}{N^{dn}} = \frac{\left(\prod_{i=1}^{n} p_i N^d \right) W_n}{N^{dn}} = \left(\prod_{i=1}^{n} p_i \right) W_n.
\end{equation}
 Letting $n\to \infty$ in \eqref{eqn leb} yields $\lambda(\Dfat)= (\prod_{i=1}^{\infty}p_i) W$. Since $\prod_{i=1}^{\infty} p_i >0$ and $W>0$ a.s.\ given non-extinction, we get the desired result. 
\end{proof}

\subsection{Proof of Theorem \ref{thm lebesgue}}

We start with a heuristic strategy for the proof. For a fixed configuration $\omega \in \Omega$, let us call a point $x$ in the unit cube \emph{conditionally connected} if the following property holds: If we change $\omega$ by retaining all cubes that contain $x$, then $x$ is contained in a connected component larger than one point. We show that for almost all points $x$ it is the case that $x$ is conditionally connected with probability 0 or 1. We define an ergodic transformation $T$ on the unit cube. The transformation $T$ enables us to prove that the probability for a point $x$ to be conditionally connected has the same value for $\lambda$-almost all $x$.  From this we can then conclude that either the set of dust points or the set of connected components contains all Lebesgue measure.

\begin{proof}[Proof of Theorem \ref{thm lebesgue}]
First, we have to introduce some notation. Let $U$ be the collection of points in $[0,1]^d$ not on the boundary of a subcube. For each $x\in U$ there exists a unique sequence $(C({\bf x}_1,\ldots,{\bf x}_n))_{n\geq 1}$ of cubes of the fractal process, where ${\bf x}_j\in J^d$ for all $j$, such that $\bigcap_{n\geq 1} C({\bf x}_1,\ldots,{\bf x}_n) = \{ x \}$. Therefore, we can define an invertible transformation $\phi : U \to (J^d)^{\IN}$ by $\phi(x) = ({\bf x}_1,{\bf x}_2,\ldots)$. For each $n \in \IN$ let $\mu_n$ be the uniform measure on $(X_n,\FF_n)$, where $X_n= J^d$ and $\FF_n$ is the power set of $X_n$. Let $(X,\FF,\mu)=\bigotimes_{n=1}^{\infty}(X_n,\FF_n,\mu_n)$ be the product space. Since $\phi : (U,\BB(U),\lambda) \to (X,\FF,\mu)$ is an invertible measure-preserving transformation, we have that $(X,\FF,\mu)$ is by definition isomorphic to $(U,\BB(U),\lambda)$. Here $\BB(U)$ denotes the Borel $\sigma$-algebra.

Next, we define the transformation $T:U\to U$, which will play a crucial role in the rest of the proof. Define the auxiliary shift transformation $T^* : X\to X$ by $T^*(({\bf x}_1,{\bf x}_2,{\bf x}_3,\ldots)) = ({\bf x}_2,{\bf x}_3,\ldots)$, for $({\bf x}_1,{\bf x}_2,\ldots)\in X$. The transformation $T^*$ is measure preserving with respect to the measure $\mu$ and also ergodic, see for instance \cite{Wal}. Let $T := \phi^{-1} \circ T^* \circ \phi$ be the induced transformation on $U$ and note that $T$ is isomorphic to $T^*$ and hence also ergodic. Informally, $T$ sends a point $x\in U$ to the point $Tx$, in such a way that the relative position of $Tx$ in the unit cube is the same as the relative position of $x$ in its level-1 cube $C({\bf x}_1)$; see Figure \ref{fig transformation}.

\begin{figure}[!ht]
\begin{center}
\includegraphics[width=6cm]{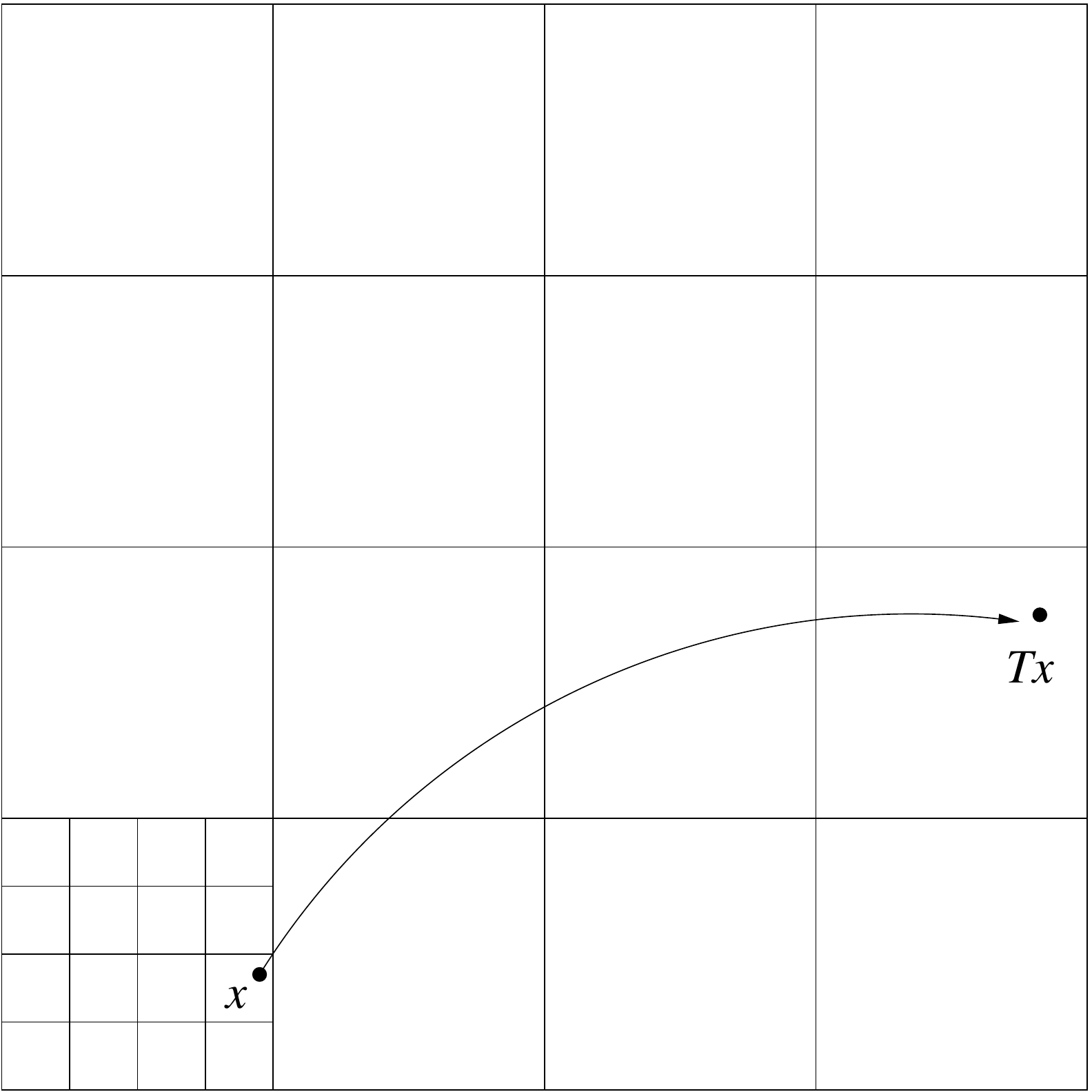}
\caption{Illustration of the transformation $T$. Note that the relative position of $x$ in the level-1 cube is the same as the relative position of $Tx$ in the unit cube.}
\label{fig transformation}
\end{center}
\end{figure}

Recall that $\omega \in \Omega$ denotes a particular realization of the fat fractal percolation process.  For $x \in U$, we define the following event.
\[
 A^x := \left\{\omega: \text{if we set } \omega(C(\textbf{x}_1,\ldots,\textbf{x}_n)) = 1 \text{ for all } n \geq 1, \text{ then } C_{\text{fat}}^{x} \neq \{x\} \right\}.
\]
In other words, $A^{x}$ consists of those configurations $\omega$ such that when we change the configuration by retaining all $C(\textbf{x}_1,\ldots,\textbf{x}_n)$, then in this new configuration, $x$ is in the same connected component as some $y \neq x$. Observe that 
\begin{equation}\label{def A^x}
 A^x \cap \{  x\in \Dfat\} =\{ x \in \Dfatc \}.
\end{equation}
It is easy to see that $A^{x}$ is a tail event. Hence, by Kolmogorov's 0-1 law we get $\IP(A^{x}) \in \{0,1\}$ for all $x\in U$.

However, a priori it is not clear that for almost all $x$ in the unit cube $\IP(A^x)$ has the same value. To this end, define the set $V:=\{x\in U: \IP(A^{x})=0\}$. We will show that $\lambda(V)\in \{0,1\}$. Recall that the relative position of $Tx$ in the unit cube is the same as the relative position of $x$ in the level-1 cube $C(\mathbf{x}_1)$. It is possible to construct a coupling between the fractal process in the unit cube and the fractal process in $C(\mathbf{x}_1)$, given that $C(\mathbf{x}_1)$ is retained,  with the following property: For every cube $C(\Ib)$ in $C(\mathbf{x}_1)$, it is the case that if $TC(\Ib)$ is retained in the fractal process in the unit cube, then $C(\Ib)$ is also retained in the fractal process in $C(\mathbf{x}_1)$, given that $C(\mathbf{x}_1)$ is retained. It is straightforward that such a coupling exists since the retention probabilities $p_n$ are non-decreasing in $n$. Hence,
\begin{equation}\label{A^Tx ineq}
 \IP(A^{Tx}) \leq \IP(A^x | C(\mathbf{x}_1) \text{ is retained}).
\end{equation}
Furthermore, since $A^x$ is a tail event, we have 
\begin{equation}\label{cond eq}
\IP(A^x)=\IP(A^x | C(\mathbf{x}_1) \text{ is retained}). 
\end{equation}
It follows from \eqref{A^Tx ineq} and \eqref{cond eq} that $\IP(A^{Tx}) \leq \IP(A^x)$ for all $x$. This implies that $V \subset T^{-1} V$. Because $T$ is measure preserving it follows that 
\[
 \lambda(V \Delta T^{-1} V) = \lambda(V \setminus T^{-1}V) + \lambda(T^{-1}V \setminus V) = 0+\lambda(T^{-1}V) - \lambda(V) = 0.
\]
Ergodicity of $T$ now yields that $\lambda(V) \in \{0,1\}$.

Suppose $\lambda(V)=0$. Then $\IP(x \in \Dfatd)=\IP(\{x\in \Dfat\} \setminus A^x)=0$ for almost all $x\in [0,1]^d$, by \eqref{def A^x}. Applying Fubini's theorem gives
\begin{eqnarray*}
 \IE \lambda (\Dfatd) &=& \int_{\Omega}\int_{[0,1]^{d}} 1_{\Dfatd}(x,\omega)d\lambda d\IP \\
&=&  \int_{[0,1]^{d}}\int_{\Omega} 1_{\Dfatd}(x,\omega) d\IP d\lambda\\
&=& \int_{[0,1]^{d}} \IP(x \in \Dfatd) d\lambda = 0.
\end{eqnarray*}
 Therefore $\lambda(\Dfatd)=0$ a.s. By Theorem \ref{thm12} we have $\lambda(\Dfatc)> 0$ a.s.\ given non-extinction.

Next suppose that $\lambda(V)=1$. Then with a similar argument we can show that $\lambda(\Dfatc)=0$ and $\lambda(\Dfatd)>0$ a.s.\ given non-extinction. 
\end{proof}

\subsection{Proof of Theorem \ref{thm charac}}

\begin{proof}[Proof of Theorem \ref{thm charac}]
(i) Suppose that $\Dfat$ has a non-empty interior with positive probability. Then we have
\begin{eqnarray*}
0 &<& \IP(\Dfat \text{ has non-empty interior}) \\
&=& \IP(\exists n, \exists \Ii_1,\ldots,\Ii_n: C(\Ii_1,\ldots,\Ii_n) \subset \Dfat) \\
&\leq& \sum_{ n, \Ii_1,\ldots,\Ii_n} \IP(C(\Ii_1,\ldots,\Ii_n) \subset \Dfat).
\end{eqnarray*}
Since we sum over countably many cubes, there must exist $n$ and $\Ii_1,\ldots,\Ii_n$ such that $\IP(C(\Ii_1,\ldots,\Ii_n) \subset \Dfat) > 0$. Hence, by translation invariance, $\IP(C(\Ii_1,\ldots,\Ii_n) \subset \Dfat) > 0$ for this specific $n$ and all $\Ii_1,\ldots,\Ii_n$. We can apply the FKG inequality to obtain $\IP(\Dfat = [0,1]^d) = \IP(C(\Ii_1,\ldots,\Ii_n) \subset \Dfat \,\forall \Ii_1,\ldots,\Ii_n) >0$. Since $\IP(\Dfat = [0,1]^d) = \prod_{n=1}^{\infty} p_n^{N^{dn}}$, this proves the first part of the theorem.

(ii) Suppose $\prod_{n=1}^{\infty} p_n^{N^n} > 0$. Then for each $x\in [0,1]^{d-1}$ we have $\IP(\{x\} \times [0,1] \subset \Dfat) \geq \prod_{n=1}^{\infty} p_n^{N^n} > 0$. Let $\lambda_{d-1}$ denote $(d-1)$-dimensional Lebesgue measure. 
Applying Fubini's theorem gives
\begin{eqnarray*}
&& \IE \lambda_{d-1}(\{x\in [0,1]^{d-1} : \{x\}\times [0,1] \subset \Dfat\}) \nonumber \\
&=&  \int_{\Omega} \int_{[0,1]^{d-1}} 1_{\{x\}\times [0,1] \subset \Dfat}\,d\lambda_{d-1} d\IP \nonumber \\
&=& \int_{[0,1]^{d-1}} \int_{\Omega} 1_{\{x\}\times [0,1] \subset \Dfat}\,d\IP d\lambda_{d-1} \nonumber \\
&=& \int_{[0,1]^{d-1}} \IP(\{x\}\times [0,1] \subset \Dfat) d\lambda_{d-1} >0. 
\end{eqnarray*}
Hence, 
\begin{equation}\label{strictly positive}
\lambda_{d-1}(\{x\in [0,1]^{d-1} : \{x\}\times [0,1] \subset \Dfat \}) > 0
\end{equation}
with positive probability. Observe that 
\[
\Dfatc \supset \bigcup_{x \in [0,1]^{d-1}: \{x\}\times[0,1]\subset \Dfat} \{x\} \times [0,1].                                                                                                                                               
\]
In particular,
\[
\lambda(\Dfatc) \geq \lambda_{d-1}(\{x\in [0,1]^{d-1} : \{x\}\times [0,1] \subset \Dfat\}).
\]
From \eqref{strictly positive} we conclude that $\lambda(\Dfatc)>0$ with positive probability. It now follows from Theorem \ref{thm lebesgue} that the Lebesgue measure of the dust set is 0 a.s.

(iii) Next assume that $\prod_{n=1}^{\infty} p_n^{N^{dn}} > 0$. For each level $n$, we have $\IP(\Dfat^n=\Dfat^{n-1}) \geq p_n^{N^{dn}}$. Since $\prod_{n=1}^{\infty} p_n^{N^{dn}} > 0$ is equivalent to $\sum_{n=1}^{\infty}( 1 - p_n^{N^{dn}}) < \infty$, we have
\[
\sum_{n=1}^{\infty} \IP(\Dfat^n\neq \Dfat^{n-1}) \leq \sum_{n=1}^{\infty} (1 - p_n^{N^{dn}}) < \infty.
\]
Applying the Borel-Cantelli lemma gives that, with probability 1, $\{\Dfat^n \neq \Dfat^{n-1}\}$ occurs 
for only finitely many $n$. Hence, with probability 1 there exists an $n$ such that $\Dfat$ can be written as the union of level-$n$ cubes.
\end{proof}

\subsection{Proof of Theorem \ref{thm pos prob}}

\begin{proof}[Proof of Theorem \ref{thm pos prob}]
$(iii) \Rightarrow (ii)$. Trivial.

$(ii) \Rightarrow (i)$.  Suppose $\IP(x \mbox{ connected to } y) > 0$ for all $x,y \in U$, for some set $U\subset [0,1]^2$ with $\lambda(U) > 0$. Fix $y\in U$. By Fubini's theorem
\begin{eqnarray*}
\IE \lambda(D^c_{\text{fat}}) &=& \int_{\Omega} \int_{[0,1]^2} 1_{D^c_{\text{fat}}}(x,\omega) d\lambda(x) d\IP(\omega) \\
&=& \int_{[0,1]^2} \int_{\Omega} 1_{D^c_{\text{fat}}}(x,\omega) d\IP(\omega) d\lambda(x) \\
&=& \int_{[0,1]^2} \IP(x\in D^c_{\text{fat}}) d\lambda(x) \\
&\geq& \int_{U\setminus \{y\}} \IP(x \mbox{ connected to } y) d\lambda(x) > 0.
\end{eqnarray*}
Hence $\lambda(D^c_{\text{fat}}) > 0$ with positive probability. By Theorem \ref{thm lebesgue} it follows that $\lambda(D^c_{\text{fat}}) > 0$ a.s.\ given non-extinction of the fat fractal process.

$(i) \Rightarrow (iii)$. Next suppose that $\lambda(D^c_{\text{fat}}) > 0$ a.s.\ given non-extinction of the fat fractal process. For points $x\in[0,1]^2$ not on the boundary of a subcube, define the event $A^x$ as in the proof of Theorem \ref{thm lebesgue}. It follows from the proof of Theorem \ref{thm lebesgue} that $\IP(A^x)=1$ for all $x\in V$, for some set $V\subset[0,1]^2$ with $\lambda(V)=1$. By \eqref{def A^x} we have for all $x\in V$
\[
\IP(x\in D^c_{\text{fat}})= \IP(x\in D_{\text{fat}})>0.
\]

Let $x\in V$. Then
\[
0 < \IP(x\in D^c_{\text{fat}}) \leq \sum_{n=1}^{\infty} \IP(\mbox{diam}(C^x_{\text{fat}}) > \textstyle{\frac{1}{n}}),
\]
where $\mbox{diam}(C^x_{\text{fat}})$ denotes the diameter of the set $C^x_{\text{fat}}$. So there exists a natural number $n_x$ such that $\IP(\mbox{diam}(C^x_{\text{fat}}) > \frac{1}{n_{x}}) > 0.$ Hence 
\[
\textstyle \IP(x \mbox{ connected to } S(x,\frac{1}{2n_{x}})) > 0,
\]
where $S(x,\frac{1}{2n_{x}})$ is a circle centered at $x$ with radius $\frac{1}{2n_{x}}$. Write $x=(x_1,x_2)$ and define the following subsets of $\IR^2$
\begin{eqnarray*}
H_1 &=& \textstyle [0,1] \times [x_2-\frac{1}{4n_x},x_2], \\
H_2 &=& \textstyle [0,1] \times [x_2,x_2+\frac{1}{4n_x}], \\
V_1 &=& \textstyle [x_1-\frac{1}{4n_x},x_1] \times [0,1], \\
V_2 &=& \textstyle [x_1,x_1+\frac{1}{4n_x}] \times [0,1].
\end{eqnarray*}
Note that for every $x \in [0,1]^2$ it is the case that at least one horizontal strip $H_i$ and at least one vertical strip $V_j$ is entirely contained in $[0,1]^2$. Define the event $\Gamma_x$ by
\begin{eqnarray*}
\Gamma_x &=& \bigcap_{i\in\{1,2\}:H_i \subset [0,1]^2} \{ \mbox{horizontal crossing in } H_i \} \\
&\cap& \bigcap_{j\in\{1,2\}:V_j \subset [0,1]^2}\{ \mbox{vertical crossing in } V_j \}.
\end{eqnarray*}
See Figure \ref{fig Gamma_x} for an illustration of the event $\Gamma_x$. From Theorem 2 in \cite{ChaChaDur} it follows that in the MFP model with parameter $p \geq p_c(N,2)$, the limit set $D_p$ connects the left-hand side of $[0,1]^2$ with its right-hand side with positive probability. It then follows from the RSW lemma (e.g.\ Lemma 5.1 in \cite{DekMee}) and the FKG inequality that $\IP_{p}(\Gamma_x)>0$. Let $A_{n}$ denote the event of complete retention until level $n$, i.e.\ $\omega(C(\Ib))=1$ for all $\Ib \in \JJ^{(n-1)}$. Since  $\prod_{n=1}^{\infty} p_n >0$ there exists an integer $n_0$  such that $p_{n} \geq p_c(N,2)$ for all $n \geq n_0$. Hence, the probability measure $\IP_{\text{fat}}(\cdot | A_{n_0}$) dominates $\IP_{p_{c}(N,2)}(\cdot)$. Since $\IP_{\text{fat}}(A_{n_0})>0$ it follows that $\IP_{\text{fat}}(\Gamma_x) > 0$.

\begin{figure}[!ht]
\begin{center}
\includegraphics[width=7cm]{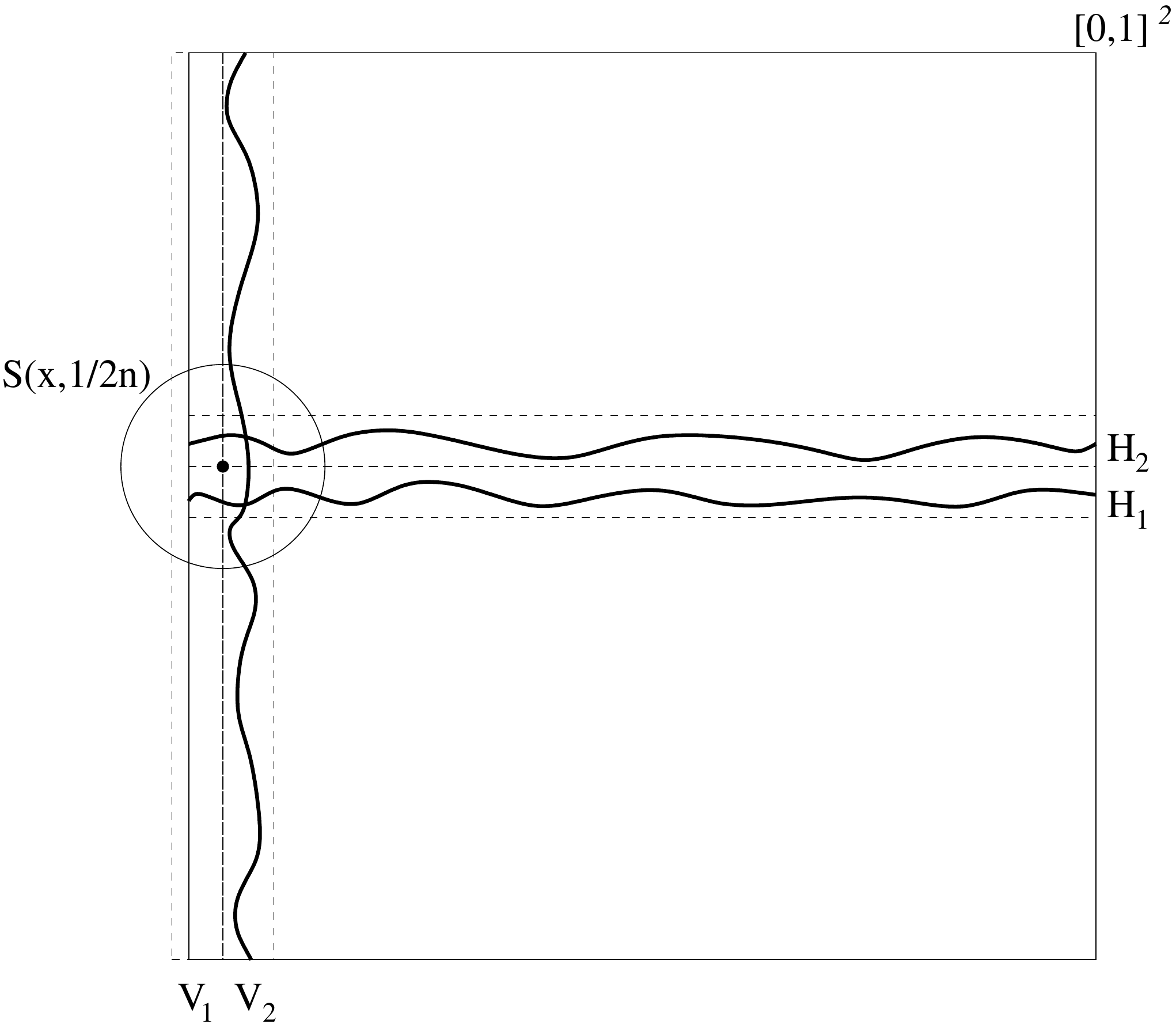}
\caption{Realization of the event $\Gamma_x$.}
\label{fig Gamma_x}
\end{center}
\end{figure}

Observe that for $x,y \in V$
\begin{eqnarray*}
\lefteqn{ \{ x \mbox{ connected to } y \} } \\
&\supset& \{ x \mbox{ connected to } S(x,\textstyle{\frac{1}{2n_x}}) \} \cap \Gamma_x \cap \{ y \mbox{ connected to } S(y,\textstyle{\frac{1}{2n_y}}) \} \cap \Gamma_y.
\end{eqnarray*}
 Since all four events on the right-hand side are increasing and have positive probability, we can apply the FKG inequality to conclude that for all $x,y\in V$ we have $\IP(x \mbox{ connected to } y) > 0$.
\end{proof}

\bibliographystyle{amsplain}
\bibliography{k-fractal}

\end{document}